\documentclass{amsart}

\usepackage{amsmath}
\usepackage{amssymb}
\usepackage{amsthm}
\usepackage{array}
\usepackage{subfigure}
\usepackage[all]{xy}
\usepackage{graphics,graphicx}
\usepackage{enumitem}

\usepackage[short,nodayofweek]{datetime}
\usepackage{hyperref}
\usepackage{nicefrac}

\usepackage{tikz}
\usetikzlibrary{matrix,arrows}

\DeclareFontFamily{OT1}{rsfs}{}
\DeclareFontShape{OT1}{rsfs}{n}{it}{<-> rsfs10}{}
\DeclareMathAlphabet{\mathscr}{OT1}{rsfs}{n}{it}

\def\bar{\overline}
\newcommand{\FF}{\mathbb F}

\newcommand{\CC}{\mathbb C}
\newcommand{\NN}{\mathbb N}

\newcommand{\gs}{\geqslant}
\newcommand{\ls}{\leqslant}

\newcommand{\chara}{\operatorname{char}}

\newcommand{\Sym}{\operatorname{Sym}}

\newcommand{\gr}{\operatorname{gr}}

\newcommand{\HH}[3]{\operatorname{H}^{#1}_{#2}(#3)}

\newcommand{\cod}{\operatorname{codim}}

\newcommand{\iso}{\cong}
\newcommand{\mdd}[1]{\; \mathrm{mod} \; ({#1})}
\newcommand{\tr}{\operatorname{Tr}}
\newcommand{\N}{\operatorname{N}}
\def\V{\mathcal{V}}

\def\I{\mcal{I}}

\def\fm{\mathfrak{m}}
\def\mcal{\mathcal}

\newcommand{\kk}{\Bbbk}
\newcommand{\kv}{{\kk[V]}}
\newcommand{\kvv}{{\kk[V^2]}}

\newcommand{\kvg}{{\kk[V]^G}}
 
\newcommand{\sv}{\mathcal{S}_V}
\newcommand{\svg}{\mathcal{S}_{V,G}}
\newcommand{\SVG}[2]{\mathcal{S}_{#1,#2}}
\newcommand{\svgw}{\mathcal{S}_{V,G_W}}
\newcommand{\svggr}[2]{\svg(>\!(1\otimes #1)(#2))}
\newcommand{\svgleq}[2]{\svg(\ls\!(1\otimes #1)(#2))}
\newcommand{\svgle}[2]{\svg(<\!(1\otimes #1)(#2))}
\newcommand{\rvg}{\mathcal{R}_{V,G}}

\def\GL{\operatorname{GL}}

\newcommand{\Htil}[1]{\operatorname{\widetilde{\operatorname{H}}}_{#1}}

\newtheorem{thm}{Theorem}[section]
\newtheorem*{thm*}{Theorem}
\newtheorem{cor}[thm]{Corollary}
\newtheorem*{cor*}{Corollary}
\newtheorem{prop}[thm]{Proposition}
\newtheorem{lem}[thm]{Lemma}

\theoremstyle{definition}

\newtheorem{conj}[thm]{Conjecture}

\theoremstyle{remark}
\newtheorem{rmk}[thm]{Remark}
\newtheorem{eg}[thm]{Example}

\title[Separating invariants and local cohomology]{Separating invariants and local cohomology}

\author[Emilie Dufresne]{Emilie Dufresne}
\address{Department of Mathematical Sciences, Durham University, Science Laboratories, South Road, Durham DH1 3LE, UK}
\email{e.s.dufresne@durham.ac.uk}

\author[Jack Jeffries]{Jack Jeffries}
\address{Department of Mathematics, University of Utah, 155 South 1400 East, Room 233, Salt Lake City, UT 84112-0090, USA}
\email{jeffries@math.utah.edu}
\date{\today}



\begin{document}

\begin{abstract}
The study of separating invariants is a recent trend in invariant theory.  For a finite group acting linearly on a vector space, a separating set is a set of invariants whose elements separate the orbits of $G$. In some ways, separating sets often exhibit better behavior than generating sets for the ring of invariants. We investigate the least possible cardinality of a separating set for a given $G$-action. Our main result is a lower bound that generalizes the classical result of Serre that if the ring of invariants is polynomial then the group action must be generated by pseudoreflections. We find these bounds to be sharp in a wide range of examples.
\end{abstract}

\thanks{This material is based upon work supported by the National Science Foundation under Grant No~0932078 000, while the authors were in residence at the Mathematical Science Research Institute in Berkeley, California, during the Spring semester of 2013. The second author was also supported in part by NSF grants DMS~0758474 and DMS~1162585.}
\maketitle


\section{Introduction}

For an action of an algebraic group on an affine variety, a \emph{separating set} is a collection of invariants which, as functions on $V$, distinguish any two points that can be distinguished by some invariant. While using invariants as a tool to distinguish orbits of a group action on a variety is a classical endeavor, this approach to invariant theory has enjoyed a resurgence of interest in its modern form, initiated by work of Derksen and Kemper \cite{DK, Kem2}. 

Throughout this paper, we focus on the case of a finite group $G$ acting linearly on a $d$-dimensional vector space $V$ over the field $\kk$. This action induces a contragredient action of the group $G$ on the polynomial ring $\kv:=\Sym(V^*)$; if $\kk$ is infinite, $\kv$ can be identified with the ring of regular functions on $V$.  We consider the \emph{ring of invariants} $\kvg:=\{f\in \kv \ |\  \forall g \in G, \,g \cdot f = f \}$. We will assume throughout that $\kk$ is algebraically closed. While the results of our paper have analogous statements over general fields (see Remark~\ref{alg_closed}), the exposition is cleaner with the assumption that $\kk$ is algebraically closed. In this setting, a separating set is a set $E\subset \kvg$ such that if, for $v, w \in V$, the orbits $G \cdot v$ and $G \cdot w$ are distinct, then there is an $h \in E$ with $h(v)\neq h(w)$; that is, a separating set is a set of invariants which separates orbits.

While the ring of invariants (or a generating set) forms a separating set, there often exist smaller and/or otherwise better-behaved separating sets --- especially in the \emph{modular} case, where $|G|$ is not invertible in $\kk$. For example, there always exist separating sets consisting of elements of degree at most $|G|$ (\cite[Corollary~3.9.14]{DK}), and polarizations of separating sets yield separating sets for vector invariants (\cite[Theorem~1.4]{DKW}). The main question we consider in this paper is: What is the least cardinality of a separating set?

Some general bounds are known. It follows from \cite[Proposition 2.3.10]{DK} that the algebra generated by a separating set, i.e., a \emph{separating algebra}, has dimension $d$, thus any separating set has at least $d$ elements. On the other hand, a secant variety argument (see \cite[Proposition~5.1.1]{Duf2}) shows that there always exists a separating set of size $2d+1$.

Since any separating algebra has dimension $d$, the existence of a separating set of size $d$ is equivalent to the existence of a polynomial separating algebra. The question of whether the ring of invariants is polynomial is very classical, and two of the cornerstone results of invariant theory largely answer this question: The Shephard-Todd Theorem  (see \cite[p.~104]{DK}) says that if $|G|$ is invertible in $\kk$ (the \emph{non-modular} case), then $\kvg$ is a polynomial ring if and only if the action of $G$ is generated by \emph{pseudoreflections} --- elements that fix a hyperplane in $V$. In this case, one calls $G$ a \emph{reflection group}. A theorem of Serre (see \cite[Proposition~3.7.8]{DK}) states that, with no hypothesis on $|G|$, if $\kvg$ is a polynomial ring, then $G$ acts as a \emph{rigid reflection group}: every isotropy subgroup is a reflection group. The problem of classifying which actions have polynomial invariant rings in the modular case remains an important open question.

In \cite[Theorem~1.1]{Duf1}, the first author extends Serre's result by showing that if there exists a polynomial separating algebra, then $G$ is a reflection group. As a corollary, in the non-modular case, there exists a polynomial separating algebra if and only if $G$ is a reflection group. The existence of a separating set of size $d$ is thus related to whether $G$ is a reflection group. Further, in \cite[Theorem~1.3]{Duf1}, the first author shows that if there is a graded separating algebra that is a complete intersection, then the action of $G$ is generated by \emph{bireflections} --- elements that fix a codimension two subspace in $V$. Consequently, if there is a separating set consisting of $d+1$ homogeneous invariants (whence the algebra it generates is a graded hypersurface and hence a complete intersection), then the action of $G$ is generated by bireflections.

In the present paper, we apply techniques of local cohomology to strengthen and extend these bounds. After reviewing some preliminary notions in Section~\ref{Prelims}, in Section~\ref{Lower}, we obtain our main result:
\begin{thm*}\label{thm-MainThm}
If there exists a separating set of size $d+r-1$, then every isotropy subgroup $G_U$ is generated by $r$-reflections. In particular, $G$ is generated by $r$-reflections.
\end{thm*}
Setting $r=1$, we obtain the following strengthening of \cite[Theorem~1.1]{Duf1}: If there exists a separating set of size $d$, then $G$ is a rigid reflection group. Our approach utilizes \`Alvarez, Garc\'ia, and Zarzuela's computation of local cohomology with support in a subspace arrangement in \cite{AGZ}. Their formula is a local cohomology analogue of the celebrated Goresky-MacPherson Formula for the singular cohomology of the complement of a real subspace arrangement (see, e.g., \cite[Theorem~1.3.8]{Wac}); in this way, one can consider our results a link between the Goresky-MacPherson Formula and the Shephard-Todd Theorem.

In Section~\ref{Ref_Groups}, we focus on rigid reflection groups. Applying techniques from poset homology, we show that the cohomological obstructions to small separating sets utilized in Section~\ref{Lower} vanish for all integers greater than $d$. While there are rigid reflection groups for which the ring of invariants is not polynomial, some of the counterexamples have been proved to have a polynomial separating algebra, e.g.~\cite[Example~3.1]{Duf1}. We pose the conjecture that there exists a polynomial separating algebra if and only if $G$ is a rigid reflection group.

In Section~\ref{Examples}, we construct a variety of examples of separating sets for which the lower bound from the main theorem is realized: that is, we construct separating sets of the minimal possible cardinality. While we do not have a specific algorithm by which we create such sets, we are able to use an idea from the first author's thesis \cite[Section~5.2]{Duf2} (the ``triangle trick'') effectively in a wide range of cases.


\section{Preliminaries}\label{Prelims}

\subsection{$r$-Reflections}\label{r-ref}

For any  subset $U$ of $V$, we define its \emph{isotropy subgroup} $G_U$ as follows:
\[G_U:=\{\sigma\in G \mid \sigma\cdot u=u, ~\forall u\in U\}\,.\]

An element $\sigma\in G$ is called an \emph{$r$-reflection} if its fixed subspace $V^\sigma$ has codimension~$r$. In particular, a $1$-reflection is a pseudoreflection, and a $2$-reflection is a bireflection. We say that $G$ is an \emph{$r$-reflection group} if it is generated by elements whose fixed space has codimension at most $r$.

A linear subspace $W\subset V$ is an \emph{$r$-reflecting subspace} if and only if $W$ has codimension $r$ in $V$ and its isotropy subgroup $G_{W}$ is non-trivial. An $r$-reflecting subspace will be called \emph{minimal} if it is not the intersection of $r'$-reflecting subspaces with $r'<r$. A group is  called a \emph{rigid $r$-reflection group} if every minimal reflecting subspace has codimension at most $r$. This is equivalent to requiring that every isotropy subgroup is an $r$-reflection group. We will say that $G$ is a \emph{(rigid) $(<\!r)$-reflection group} if there exists an $r'<r$ such that $G$ is a (rigid) $r'$-reflection group. For $r=1$ we will say \emph{(rigid) reflection group} instead of (rigid) 1-reflection group. 

In the non-modular case, it follows from the Shephard-Todd Theorem and Serre's Theorem that every reflection group is a rigid reflection group. For $r>1$, the condition of being a rigid $r$-reflection group is stronger than that of being an $r$-reflection group. For a concrete example, let $V$ be a $(2n+1)$-dimensional vector space over $\CC$ with basis $u_1,\dots,u_n,v_1,\dots,v_n,w$ and let $G:=C_2\times C_2=\langle \alpha, \beta \rangle$ act on $V$ by
\[
\begin{aligned}
&\alpha(u_i)=-u_i \qquad &\beta(u_i)&=u_i  \qquad &\textrm{ for } i=1,\ldots, n\,,\\
&\alpha(v_i)=v_i \qquad  &\beta(v_i)&=-v_i  \qquad &\textrm{ for } i=1,\ldots, n\,,\\
&\alpha(w)=-w\, \qquad &\beta(w)&=-w\,.  \qquad  &
\end{aligned}
\]
Here $G$ is generated by $(n+1)$-reflections, but $\langle \alpha\beta \rangle$ is an isotropy subgroup generated by a $(2n)$-reflection, thus $G$ is not a rigid $(<\!2n)$-reflection group.

\subsection{The Separating Variety}

The \emph{separating variety} $\svg$ is a closed subvariety of the product $V\times V$ that completely determines the equivalence relation induced by $\kvg$ on $V$. More precisely, we have
\begin{align*}
\sv:&=\{(u,v)\in V\times V\mid f(u)=f(v), \textrm{ for all } f\in \kvg \}\\
&=\V_{V\times V}(f\otimes 1-1\otimes f\mid f\in\kvg)\,.
\end{align*}

A separating set can then be characterized as a subset $E\subset\kvg$ that cuts out the separating variety in $V \times V$, that is, such that $\V_{V\times V}(f\otimes 1-1\otimes f\mid f\in E)=\svg$. In ideal-theoretic terms,

\begin{prop}\cite[Corollary~2.6]{Kem2}\label{SepDefn} A set of invariants $\{f_1,\dots,f_t\}$ is a separating set for $G$ acting on $V$ if and only if 
\[\sqrt{(f_1\otimes 1 -1\otimes f_1,\dots, f_t \otimes 1 -1 \otimes f_t)}=\sqrt{(h\otimes 1 -1\otimes h\mid h\in\kvg)}=: \I(\svg)\,.\]
\end{prop}

For actions of finite groups, the invariants actually separate orbits (see, for example, \cite[Lemma~2.1]{DKW}) and so the separating variety coincides with the graph of the action
\[\Gamma_{V,G}:=\{(v,\sigma\cdot v) \mid v\in V,~\sigma\in G\}\,.\]
This provides significant geometric insight into $\svg$:

\begin{lem}[{c.f. \cite[Proposition~3.1]{Duf1}}]\label{SV}\label{ints_are_isotrops} Let $G$ be a finite group acting linearly on $V$.
\begin{itemize}[leftmargin=9mm]
\item[(a)]  The separating variety has an irreducible decomposition of the form
\[\svg = \bigcup_{\sigma\in G} (1\otimes \sigma)(V)\]
with each $(1\otimes \sigma)(V)$ a linear subspace isomorphic to $V$.
\item[(b)] If $\sigma, \tau \in G$, then ${(1\otimes \sigma)(V)\cap (1\otimes \tau)(V)}=(1\otimes \tau)( V^{\tau^{-1}\sigma})$, which has dimension equal to that of the subspace fixed by $\tau^{-1}\sigma$ in $V$. Every non-empty intersection of components $(1\otimes\sigma)(V)$ with  $\sigma \in G$  is of the form ${(1\otimes \gamma)(V^H)}$, where $H\leq G$ is an isotropy subgroup and $\gamma\in G/H$.
\end{itemize}
\end{lem}

\begin{rmk}\label{alg_closed} The assumption that $\kk$ is algebraically closed is essential in Proposition~\ref{SepDefn}. However, one may obtain results in the non-algebraically closed case by considering a \emph{geometric separating set}: for $G$ finite, this is a subset of $\kvg$ that separates orbits of $G$ in $V \otimes_{\kk} \bar{\kk}$ (see \cite[Section~2]{Duf1}).
By \cite[Theorem~2.1]{Duf1}, a geometric separating set is characterized by the ideal-theoretic equality in Proposition~\ref{SepDefn}. Accordingly, the results of Section~\ref{Lower} hold for $\kk\neq\bar{\kk}$ if one replaces the phrase ``separating set'' with ``geometric separating set.'' Further, since $\kvg$ is a geometric separating set, Corollary~\ref{emb_dim} holds verbatim for all $\kk$.
\end{rmk}

\subsection{Posets}

For an arrangement of linear subspaces $X\subset \mathbb{A}^m$, let $P(X)$ denote the \emph{intersection poset} of $X$: the collection, ordered by inclusion, of linear subspaces that occur as intersections of components of $X$. For $p\in P(X)$, the \emph{interval} $P(>\!p)$ is the subposet of $P(X)$ consisting of elements containing $p$. One defines $P(<\!p), P(\gs\!p),$ and $P(\ls\!p)$ analogously. The reduced homology of a poset $P$ with coefficients in $\kk$ will be denoted by $\Htil{\bullet}(P; \kk)$: this is the reduced simplicial homology of the simplicial complex whose vertices are elements of the poset, and whose faces are the chains.

In our setting, for a linear action of a finite group, the separating variety $\svg$ is
a subspace arrangement. By abuse of notation, we will also denote its intersection poset by $\svg$. Note that if $W\subseteq V$ is a subspace, then $\svggr{1}{W}\iso\svgw(>\!(1\otimes 1)(W))$.

We will also consider the poset $\rvg$ of $r$-reflecting subspaces (all possible $r$'s). The two posets $\svg$ and $\rvg$ are related by the following lemma.

\begin{lem}\label{rvg} For any $\sigma \in G$, the interval $\svgle{\sigma}{V}$ is isomorphic to $\rvg$.
\end{lem}
\begin{proof} The map on $\Gamma_{V,G}$ given by applying $\sigma$ to the second coordinate is an isomorphism, thus $\svgleq{\sigma}{V}\iso \svgleq{1}{V}$. Now,
\begin{align*}
(1\otimes 1)(V) \cap (1 \otimes \sigma_1)(V) &\cap \cdots \cap (1 \otimes \sigma_m)(V) \\
&= \{ (v,v)\, | \,v=\sigma_1(v)=\cdots=\sigma_m(v) \}\\
&= (1\otimes 1)(V^{\langle \sigma_1, \dots, \sigma_m \rangle})\,,
\end{align*}
so that the intersections of components of $\svg$ contained in $(1\otimes 1)(V)$ coincide with the diagonal embeddings of reflecting subspaces.
\end{proof}

It is worth noting that the order on $\rvg$ used here is dual to that most commonly used in the literature on subspace arrangements.

\subsection{Local Cohomology}

For the convenience of the reader unfamiliar with local cohomology, we give a quick review with an eye towards the main fact we will employ. A welcoming source on local cohomology which includes the material below is \cite{IILLMMSW}.
For an ideal $I$ in a commutative noetherian ring $R$ and an $R$-module $M$, the \emph{$I$-torsion} part of $M$ is
\[\Gamma_I(M)=\{m\in M \ | \ I^t m = 0 \ \text{for some}\ t\in \NN\}\,.\]
The assignment $\Gamma_I(-)$ is easily checked to form a left-exact functor (with maps given by restriction), and its right-derived functors are defined as the \emph{local cohomology functors with support in $I$,} denoted $\HH{i}{I}{-}$. Since $\Gamma_I(-)=\Gamma_{J}(-)$ if $\sqrt{I}=\sqrt{J},$ we also have $\HH{i}{I}{-}=\HH{i}{J}{-}$.

Given a generating set $I=(f_1,\dots,f_t)$, the local cohomology of $I$ can also be computed as the cohomology of the \v{C}ech complex:
\[\HH{i}{I}{M}=\operatorname{H}^i  \Big( \quad 0\rightarrow M \rightarrow \bigoplus_j M_{f_j} \rightarrow \bigoplus_{j<j'} M_{f_j f_{j'}} \rightarrow \cdots \rightarrow M_{f_1\cdots f_t} \rightarrow 0 \quad \Big)\ ,\]
where the maps on each component are $\pm 1$ times the natural maps, with the signs chosen so that the sequence above forms a complex. Consequently, if $\HH{i}{I}{R}\neq 0$ and $f_1,\dots, f_t$ generates $I$ up to radical, we necessarily have $t\gs i$, since the \v{C}ech complex must have at least $i$ terms if its $i^\text{th}$ cohomology is non-zero.


\section{Lower bounds on the size of separating sets}\label{Lower}

In this section, we give a lower bound on the size of a separating set for a ring of invariants of a finite group. We reiterate the assumption that $\kk$ is algebraically closed; see Remark~\ref{alg_closed} for the non-algebraically closed case. The following lemma will be key to our applications.

\begin{lem}\label{connect_reflect} The separating variety is connected in codimension $\ls\! r$ if and only if the action of $G$ is generated by $(\ls\! r)$-reflections.
\end{lem}
\begin{proof}
By Lemma~\ref{SV}~(a), the separating variety $\svg$ is connected in codimension $\ls\! r$ if and only if, for any $\sigma,\sigma' \in G$, there is a sequence of components
\[(1\otimes {\sigma})(V)=(1\otimes {\sigma_0})(V)\ ,\ (1\otimes {\sigma_1})(V) \ ,\ \dots \ ,\ (1\otimes {\sigma_r})(V)=(1\otimes {\sigma'})(V) \] 
such that $(1\otimes {\sigma_i})(V) \cap (1\otimes \sigma_{i+1})(V)$ has codimension $\ls r$. By Lemma~\ref{SV}~(b), $\dim {(1\otimes {\sigma_i})(V) \cap (1\otimes \sigma_{i+1})(V)} = \dim V^{\sigma_{i+1}^{-1}\sigma_{i}}$. Thus, $\svg$ is connected in codimension $\ls\! r$ if and only if for any $\sigma,\sigma' \in G$ there exist $(\ls\! r)$-reflections 
\[\tau_1=\sigma_0^{-1}\sigma_1\ ,\ \tau_2=\sigma_1^{-1}\sigma_2\ ,\ \dots \ ,\ \tau_r=\sigma_{r-1}^{-1}\sigma_r\]
 such that $\sigma=\tau_1\cdots \tau_r\sigma'\,.$ But this just means that $G$ is generated by $(\ls\! r)$-reflections.
\end{proof}

We first note that a connectedness theorem of Grothendieck allows for the following generalization of \cite[Theorem~1.1]{Duf1}.

\begin{prop}\label{weaker_bound} If there exists a separating set of size $d+r-1$, then the action of $G$ is generated by $(\ls\! r)$-reflections.
\end{prop}
\begin{proof} By Proposition~\ref{SepDefn}, if there is a separating set of size $d+r-1$, then $\I(\svg)$ is set-theoretically defined by $d+r-1$ equations. By \cite[Expos\'e XIII, Th\'eor\`eme~2.1]{SGA2}, if $\I(\svg)$ can be set-theoretically cut out by $d+r-1$ or fewer equations, then $\svg$ is connected in codimension $\ls\! r$. Then, by Lemma~\ref{connect_reflect}, $G$ is generated by $(\ls\! r)$-reflections.
\end{proof}

A stronger result can be obtained by examining the local cohomology with support in $\I(\svg)$. Local cohomology with support in a subspace arrangement is studied by \`Alvarez, Garc\'ia, and Zarzuela in \cite{AGZ}. Following along the lines of Bj\"orner and Ekedahl's computation of $\ell$-adic cohomology of such spaces, they establish a Mayer-Vietoris spectral sequence for local cohomology and show that it degenerates for subspace arrangements, thus obtaining a Goresky-MacPherson analogue in local cohomology. In particular, their formula provides a combinatorial characterization of the vanishing and non-vanishing of the local cohomology modules.

\begin{thm}\label{lc_arrangements}
\begin{itemize}
\item[(a)]\cite[p.~39]{AGZ}, \cite[Theorem~2.1]{Lyu} If $I_1,\dots,I_t\subset R$ are ideals, and $M$ an $R$-module, then there is a Mayer-Vietoris spectral sequence
\[ E_1^{-p,q}=\bigoplus_{i_0<\cdots<i_p} \!\HH{q}{I_{i_0}+\cdots+I_{i_p}}{M} \quad \Longrightarrow \quad \HH{q-p}{I_1 \cap \cdots \cap I_t}{M}\,.\]
\item[(b)]\cite[Corollary~1.3]{AGZ} If $I_1,\dots,I_t\subset R$ are ideals of linear subspaces in a polynomial ring, then the spectral sequence above degenerates at $E_2$, and for all $q\gs 0$ there is an associated graded module of the local cohomology module $\HH{q}{I_1 \cap \cdots \cap I_t}{R}$ with
\[
\qquad\qquad\gr\big(\HH{q}{I_1 \cap \cdots \cap I_t}{R}\big) \cong
\bigoplus_{p\in P}\,\Big[ \HH{\cod(p)}{I(p)}{R} \otimes_{\kk} \Htil{\cod(p)-q-1} (P(>\!p);\kk)\Big]\,,
\]
where $P$ is the intersection poset of $\V(I_1 \cap \cdots \cap I_t)$.
\end{itemize}
\end{thm}

With this description of the local cohomology in hand, we obtain the following strengthening of Proposition~\ref{weaker_bound}.

\begin{thm}\label{nonvanishing}
Let $r_1,\ldots, r_s$ be the codimensions of minimal reflecting subspaces. Then $\HH{d+r_i-1}{\I(\svg)}{\kvv}\neq 0$. In particular, if $r$ is the maximal codimension of a minimal reflecting subspace, then every separating set has size at least $d+r-1$.
\end{thm}
\begin{proof} 
Let $W \subset V$ be a minimal $r$-reflecting subspace in the sense of Subsection~\ref{r-ref}. Note that
 \[\svggr{1}{W}\cong \svgw(>\!(1 \otimes 1)(V^{G_W}))\,.\]
  The latter poset is connected if and only if $\svgw$ is connected in codimension $<\!r$. By Lemma~\ref{connect_reflect}, this is the case if and only if $G_W$ is generated by $(<\!r)$-reflections. Since $W$ is minimal, $G_W$ is not generated by $(<\!r)$-reflections: if $G_W=\langle g_1 , \dots , g_s \rangle$ with each $g_i$ an $(<\!r)$-reflection, one may write $W = \bigcap_{i=1}^s V^{\langle g_i \rangle}$, expressing $W$ as the intersection of larger reflecting subspaces. Thus, \[\Htil{0}\big(\svggr{1}{W};\kk \big)\neq 0\,.\]
 Theorem~\ref{lc_arrangements}~(b) applies to show that $\HH{d+r-1}{\I(\svg)}{\kvv}\neq 0$. Thus, $\I(\svg)$ cannot be set-theoretically defined by $d+r-1$ or fewer equations, and by Proposition~\ref{SepDefn}, any separating set has size at least $d+r-1$.
\end{proof}

\begin{cor}\label{emb_dim} If $r$ is the maximal codimension of a minimal reflecting subspace, then the embedding dimension of $\kvg$ is at least $d+r-1$.
\end{cor}
\begin{proof}
This follows immediately from Theorem~\ref{nonvanishing} since a minimal generating set is a separating set. Alternatively, one may argue by using Proposition~\ref{weaker_bound} to conclude that the embedding codimension is at least $r$ if $G$ is not a $(<\! r)$-reflection group, and applying \cite[Theorem~A]{Kem}, according to which the embedding codimension (referred to in \emph{ibid.} as the polynomial defect) does not increase when passing to the invariants of an isotropy subgroup.
\end{proof}

\begin{rmk} In a recent work of Reimers \cite[Theorem~2.4]{Rei}, the statement of Lemma~\ref{connect_reflect} is established in the more general setting where $G$ acts on a variety that is connected in codimension $\ls\! r$. This result is then applied to study the depth of schemes defining the separating variety of the action --- particularly,  in terms of local cohomology, the least $i$ for which $\HH{i}{\fm}{R/J}\neq 0$ for some $J$ with $\sqrt{J}=\I(\svg)$. In characteristic $p>0$, the vanishing of these local cohomology modules is related to the vanishing of those considered above by Peskine and Szpiro's vanishing theorem \cite[Remarque~p.~110]{PS}.
\end{rmk}

\begin{rmk} It follows from the Hartshorne-Lichtenbaum vanishing theorem \cite[Theorem~3.1]{Har} that $\HH{2d}{\I(\svg)}{\kvv}=0$. This can also be deduced from Theorem~\ref{lc_arrangements}. Indeed, the only potential element of the poset $\svg$ of codimension $2d$ is ${(1\otimes 1)(V^G)}$, and this occurs only if $V^G$ is the origin. As $\svggr{1}{V^G}$ is non-empty, $\Htil{-1}(\svggr{1}{V^G};\kk)=0$, and we are done.
\end{rmk}


\section{Rigid Reflection Groups}\label{Ref_Groups}

In this section, we focus on rigid reflection groups. In this situation, every minimal reflecting subspace is a hyperplane, and in particular, the arrangement of reflecting subspaces $\rvg$ is a hyperplane arrangement. Recall that a simplicial complex is \emph{pure} if each of its maximal facets have the same dimension. A pure simplicial complex is \emph{shellable} if there is a linear ordering of its maximal facets (a \emph{shelling}) $F_1, F_2 , \dots , F_t$ such that $F_i \cap \bigcup_{j<i} F_j$ is pure of codimension $1$; we call a poset \emph{shellable} if its order complex is pure and shellable. The salient fact we use is the following well-known tool in combinatorial topology; see, e.g., \cite[Subsection~3.1]{Wac}.

\begin{prop}\label{shellableCM}
The only non-vanishing homology of a shellable poset is in the dimension of the poset.
\end{prop}

We refer to \cite[Subsection~3.2]{Wac} for the notions  and facts from poset topology used in the proof of the following lemma. This lemma is undoubtedly previously known, but we were unable to find it in the literature in the form needed for the subsequent theorem.

\begin{lem}\label{any_facet} If $G$ acts on $V$ as a rigid reflection group, and $H$ is a reflecting hyperplane, then there exists a shelling of $\rvg$ starting with a facet containing $H$.
\end{lem}
\begin{proof} 
Note first that it is equivalent to find such a shelling of the dual $\rvg^*$ of $\rvg$.
Since $\rvg^*$ is the standard poset of a hyperplane arrangement, it is a geometric lattice, whose atoms are the reflecting hyperplanes. For any ordering of these atoms $H=H_1, H_2, \dots , H_t$, label each edge of the Hasse diagram, $(x,y)$, where $y$ covers $x$, with the least integer $i$ such that the join of $x$ and $H_i$ is $y$. This is an EL-labelling, so the associated lexicographic ordering on the maximal chains is a shelling, and the first facet of this shelling contains $H$.
\end{proof}

\begin{thm}\label{shell_rigid} If $G$ acts on $V$ as a rigid reflection group, then the intersection poset of $\svg$ is shellable.
\end{thm}
\begin{proof} Order the elements of $G$ 
\[1=\sigma_0 \ , \ \sigma_1 \ , \ \dots \ , \ \sigma_{|G|-1} \]
so that for each $j > 0$ there is some $ i <j $ such that $\sigma_i^{-1}\sigma_j$ is a reflection. We then construct a shelling inductively as follows.

First, by the identification $\svgleq{1}{W}\iso \rvg$ from Lemma~\ref{rvg}, list the facets in a shelling of $\svgleq{1}{V}$. Then, for $j>0$, given a list of the facets in 
\[\bigcup_{j' < j} \svgleq{\sigma_{j'}}{V}\]
such that each subsequent facet intersects the union of the others in pure codimension 1, choose an $i <j$ such that $\sigma_i^{-1}\sigma_j$ is a reflection. By Lemmas~\ref{rvg} and~\ref{any_facet}, list the facets in a shelling of $\svgleq{\sigma_j}{V}$ that starts with a facet $F_j$ containing a facet of
\[\svgleq{\sigma_i}{V} \cap \svgleq{\sigma_j}{V} \ = \  \svgleq{ \sigma_j}{V^{\sigma_i^{-1}\sigma_j}}\,.\]
As this is a codimension 1 subposet of ${\bigcup_{j' < j} \svgleq{\sigma_{j'}}{V}}$, the facet $F_j$ intersects the union of previously listed faces in codimension 1. Continue with the list of facets in the chosen shelling of ${\svgleq{\sigma_j}{V}}$. 

Iterating this procedure for all $j=0,\dots,|G|-1$ produces a shelling of $\svg$.
\end{proof}

As a consequence, we find that our method from Theorem~\ref{nonvanishing} does not provide sharper bounds for rigid reflection groups.

\begin{cor} If $G$ acts on $V$ as a $d$-dimensional rigid reflection group, then $\HH{t}{\I(\svg)}{\kvv}=0$ for all $t\neq d$.
\end{cor}
\begin{proof} Since $G$ is a rigid reflection group, $G_W$ is a reflection group for each isotropy subgroup $G_W$. Then, by Theorem~\ref{shell_rigid} and Proposition~\ref{shellableCM}, we find that \[{\Htil{i}(\svggr{1}{V^{G_W}}; \kk)=0} \quad \text{for all} \quad i \neq\cod(V^{G_W})-1\,.\] 
Since 
\[\svggr{1}{V^{G_W}} \iso \svggr{\tau}{V^{G_W}}\]
 for any $\tau$, by Lemma~\ref{ints_are_isotrops}, we have $\Htil{i}(\svg(>\!p); \kk)=0$ for all ${i\neq\cod(p)-1}$ and all $p$ in the intersection poset. The result follows by Theorem~\ref{lc_arrangements}.
\end{proof}

\begin{conj}\label{rigid_poly}
There exists a separating set of size $d$ (that is, there exists a polynomial separating algebra) if and only if $G$ is a rigid reflection group.
\end{conj}

The following example shows that the bounds in Theorem~\ref{nonvanishing} are not necessarily sharp if $G$ is not a reflection group.

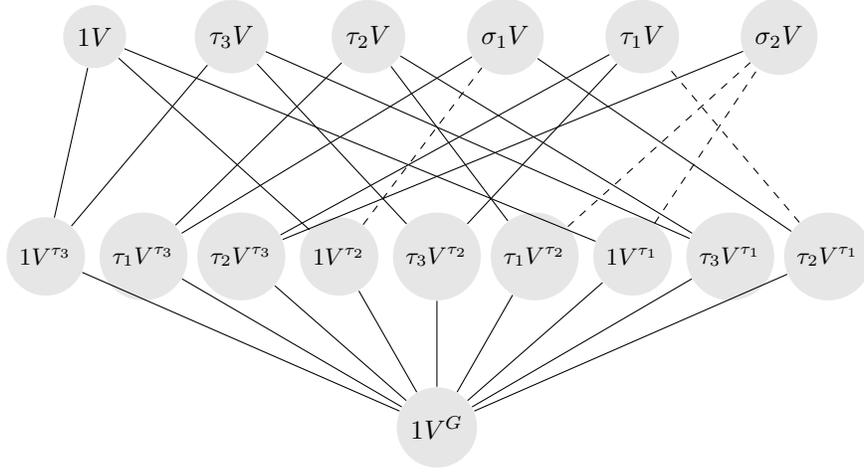
\begin{figure}
\begin{center}
\begin{tikzpicture}
  [scale=.65,auto=left,every node/.style={circle,fill=gray!20}]
  \node (na) at (3,6.5) {$ 1 V$};
  \node (nb) at (5.8,6.5)  {$\tau_3  V$};
  \node (nc) at (8.6,6.5)  {$\tau_2  V$};
  \node (nd) at (11.4,6.5) {$\sigma_1  V$};
  \node (ne) at (14.2,6.5)  {$\tau_1  V$};
  \node (nf) at (17,6.5)  {$\sigma_2 V$};
\node (n1) at (2, 2) {\small $1 V^{\tau_3}$};
\node (n2) at (4, 2) {\small $\tau_1 V^{\tau_3}$};
\node (n3) at (6, 2) {\small $\tau_2 V^{\tau_3}$};
\node (n4) at (8, 2) {\small $1 V^{\tau_2}$};
\node (n5) at (10, 2) {\small $\tau_3 V^{\tau_2}$};
\node (n6) at (12, 2) {\small $\tau_1 V^{\tau_2}$};
\node (n7) at (14, 2) {\small $1 V^{\tau_1}$};
\node (n8) at (16, 2) {\small $\tau_3 V^{\tau_1}$};
\node (n9) at (18, 2) {\small $\tau_2 V^{\tau_1}$};
\node (n0) at (10,-1.5) {$1 V^G$};

  \foreach \from/\to in {n1/na,n1/nb,n2/nd,n2/nc,n3/nf,n3/ne,
n4/na,n5/nb,n5/ne,n6/nc,
n7/na,n8/nb,n8/nc,n9/nd,
n1/n0,n2/n0,n3/n0,n4/n0,n5/n0,n6/n0,n7/n0,n8/n0,n9/n0}
    \draw (\from) -- (\to);

  \foreach \from/\to in {n4/nd,n6/nf,n7/nf,n9/ne}
    \draw (\from) -- (\to) [dashed];

\end{tikzpicture}
\end{center}
\caption{The intersection poset of the separating variety of the permutation representation of $\mathrm{S}_3$.}
\label{SVG_S3}
\end{figure}

\begin{eg} Let $G$ be the symmetric group on three letters, with elements 
\[1\ , \ (12)=\tau_3 \ , \ (13)=\tau_2 \ , \ (23)= \tau_1 \ , \ (132) = \sigma_1 \ , \ (123) = \sigma_2 \,.\]
Let $V$ be its standard three-dimensional permutation representation. Let ${W=V^{\oplus n}}$ with $G$ acting diagonally. The group $G$ acts on $V$ as a rigid reflection group, and its action on $W$ is as a rigid $n$-reflection group. Note that the intersection poset of $\SVG{W}{G}$ is isomorphic to that of $\svg$, since for any subgroup $H$ of $G$, one has $W^H=(V^H)^{\oplus n}$. This intersection poset is depicted in Figure~\ref{SVG_S3} where $gV$ is shorthand for $(1\otimes g)(V)$, and similarly for $g V^h$.

The complex of $\SVG{V}{G}(>\!(1\otimes 1)( V^G) )$ is a graph, namely the subgraph of Figure~\ref{SVG_S3} obtained by deleting the bottom vertex. Its homology may be computed by first contracting a maximal tree, depicted with solid lines; the resulting graph consists of the four dotted edges looped around a single point. We thus have
\[\Htil{1}(\SVG{W}{G}(>\!(1\otimes 1)(W^G));\kk)\iso\Htil{1}(\svg(>\!(1\otimes 1)(V^G));\kk) \iso \kk^4 \,.\]
By Theorem~\ref{lc_arrangements}, $\HH{5n-2}{\I(\SVG{W}{G})}{\kk [W^2]}\neq 0$, so, as in the argument of Theorem~\ref{nonvanishing}, we conclude that any separating set for $W$ has at least $5n-2$ elements. Note that the bound provided by Theorem~\ref{nonvanishing} for $W$ is $4n-1$.
\end{eg}


\section{Examples of separating sets of minimal size}\label{Examples}

Below, we present a variety of examples of separating sets that realize the lower bound in Theorem~\ref{nonvanishing}, thereby showing that the bound is sharp for these actions and that the found separating sets are of minimal size. First, we review an example from the first author's thesis:

\begin{prop}\cite[Proposition~5.2.2]{Duf2}\label{diag_cyc} Let $G=\langle \sigma \rangle$ be the cyclic group of order $m$, and suppose $\kk$ contain $\zeta$, a primitive $m^\text{th}$ root of unity. Let $G$ act diagonally on $\kv$ by the rule
\[ \sigma(x_i)=\zeta^{d_i} x_i\]
where $1=d_1 | d_2 | \cdots | d_n | m$. Then there is a separating set for $\kvg$ of order $2n-1$. 
\end{prop}

For this construction, a separating set of monomials $u_{i,j}: 1\ls i \ls j \ls n$ is first identified; see \cite[Proposition~5.2.2]{Duf2} for precise formulas for the $u_{i,j}$. The terms naturally align in a triangle. It is then shown that the values of the invariants $u_{i,j}$ can be recovered from the diagonal sums $S_k=\sum_{i+j=k} u_{i,j}$ of the triangle. This ``triangle trick'' is used in many of the examples below.

It is worth noting that Proposition~\ref{diag_cyc} includes as a special case the $m^\text{th}$ Veronese subring of a polynomial ring of dimension $n$, for $\chara(\kk)\not| \,m$.

\subsection{Indecomposable representations of cyclic groups of prime order}

In this subsection we construct separating sets of minimal size for the indecomposable modular representations of a cyclic group of prime order, equal to the characteristic of the field $\kk$. Our argument is greatly inspired by Sezer's iterative construction of a separating set (see \cite{Sezer-Cp}) and uses the triangle trick mentioned above. After an appropriate change of basis, any indecomposable representation of a cyclic group of prime order will be given by a Jordan block of size at most $p$. We may further choose a basis so that the action on the coordinate ring $\kk[x_1,\ldots,x_n]$ with $n \ls p$ is as follows:
\[\begin{aligned}
\sigma\cdot x_i &= x_i+x_{i+1}\,, \textrm{ for } i=1,\ldots, n-1\,,\\
\sigma\cdot x_n &= x_n\,.\end{aligned}\]
One way to construct some invariants is to take norms (orbit products)  and traces (orbit sums) of elements: in fact, by \cite[Theorem~3]{NS}, for representations of $p$-groups, norms and transfers will form a separating set. For $f\in \kk[V]$, the \emph{norm of $f$} is the orbit product  $\N(f):=\prod_{i=0}^{p-1} (\sigma^i\cdot f)$ and the \emph{trace of $f$} is the orbit sum $\tr(f):=\sum_{i=0}^{p-1} (\sigma^i\cdot f)$.

\begin{prop}  Let $V_n$ be the $n$-dimensional indecomposable representation of the cyclic group of order $p$. The set $S_n$ of the sum of the elements appearing on the diagonal of the following triangle forms a separating set.
\begin{equation} \label{eqn-triangleCp}
\begin{array}{cccccc}
\N(x_1)  &  \tr(x_1x_2^{p-1})  &  \tr(x_1x_3^{p-1})  & \cdots     & \tr(x_1x_{n-1}^{p-1})   & \\
              &  \N(x_2)                   &  \tr(x_2x_3^{p-1})  & \cdots     & \tr(x_2x_{n-1}^{p-1})   & \\
              &                                  &    \N(x_3)                  & \ddots     &   \vdots                           &     \\
              &                                  &                                  &                 &  \N(x_{n-1})                     &   \\
              &                                  &                                  &                 &                                         &  x_n^p \\
                                                                                             \end{array}
\end{equation}
\end{prop}
\begin{proof}
We proceed by induction on $n$. For $n=2$, we have $\kk[x_1,x_2]^{C_p}=\kk[\N(x_1),x_2]$. As $x_2$ and $x_2^p$ separate the same points, we are done.

Now, suppose $n\gs 2$. If $x_n^p=0$, then $x_n=0$ and the triangle (\ref{eqn-triangleCp}) reduces to the triangle for $V_{n-1}$. Thus the sum of the diagonals separate by the induction hypothesis.

Now suppose that $x_n\neq0$. For $i\gs n-2$, the coefficient of $x_i$ in $\tr(x_i x_{n-1}^{p-1})$ is
\[\begin{array}{l}
x_{n-1}^{p-1} +\sum_{l=0}^{p-1} { {p-1} \choose j} x_{n-1}^j  x_n^{p-1-j} \left( 1+2^{p-1-j}+\cdots +(p-1)^{p-1-j}\right)\\
=x_{n-1}^{p-1} - x_{n}^{p-1} -x_{n-1}^{p-1}= - x_{n}^{p-1}.\end{array}\]
Indeed, in characteristic $p$, one has $ \left( 1+2^{p-1-j}+\ldots +(p-1)^{p-1-j}\right)=-1$ for $j=0$ or $j=p-1$ and zero otherwise. It follows that
\[\kk[x_1,\ldots,x_n, x_n^{-1}]=\kk[\tr(x_1x_{n-1}^{p-1}), \ldots, \tr(x_{n-2}x_{n-1}^{p-1}), x_{n-1},x_n,x_n^{-1}]\,.\]
Taking invariants, we then have:
\[\kk[x_1,\ldots,x_n,x_n^{-1}]^{C_p}=\kk[\tr(x_1x_{n-1}^{p-1}), \ldots, \tr(x_{n-2}x_{n-1}^{p-1}), \N(x_{n-1}),x_n,x_n^{-1}]\,.\]
That is, the invariants which appear in the one before last column of the triangle (\ref{eqn-triangleCp}) generate up to dividing by some power of $x_n$. Now we need only explain how to get these from $S_n$. The bottom two, $\N(x_{n-1})$ and  $\tr(x_{n-2}x_{n-1}^{p-1})$, are in $S_n$. As any term in the triangle can be expressed as a polynomial, up to dividing by a power of $x_{n-1}$, in elements of $S_n$ lying either on the same row or below, we can express the remaining elements of $S_n$ in terms of the sums of the diagonals.
\end{proof}

\subsection{Vector Invariants of $V_2$}

Let $V_2$ denote the 2-dimensional indecomposable representation of $C_p$ as above. We consider the diagonal representation of $C_p$ on $V_2^{\oplus n}$. Let $x_1, y_1, \dots, x_n, y_n$ be a choice of coordinates on $V_2^{\oplus n}$ such that ${\sigma \cdot x_i = x_i}$, and ${\sigma \cdot y_i = x_i + y_i}$. The ring of invariants is generated by
\begin{align*}
 &x_i\,, &1\ls i \ls n &\\
&u_{i,i}=\, \N(y_i)\,=\,y_i^p-x_i^{p-1}y_i\,, &1\ls i \ls n &\\
&u_{i,j}=\,x_i y_j - x_j y_i\,, &1\ls i < j \ls n &\\
&\tr^{C_p}(y_1^{a_1}\cdots y_n^{a_n})\,, &a_i<p\,,\, \Sigma a_i \gs 2p-2\,. &
\end{align*}
By \cite[Corollary~3.9.14]{DK}, the invariants of degree less than $|G|=p$ form a separating set: in particular, the generators 
\[x_i: 1\ls i \ls n \quad \text{ and } \quad u_{ij}: 1 \ls i \ls j \ls n\]
 form a separating set. Note that we have the relations
\begin{align*}
&x_i u_{j,k} - x_j u_{i,k} + x_k u_{i,j}\, &\forall i<j<k,\\
&x_i u_{j,j} - x_j u_{i,i} + x_i^{p-1} x_j^{p-1} u_{i,j} - u_{i,j}^p\, &\forall i<j\,.
\end{align*}
 Set $S_{\ell}=\sum_{i+j=\ell} u_{i,j}$ for all $2\ls \ell \ls 2n$. Remark that the $S_{\ell}$ correspond to the diagonal sums of the triangle consisting of the $u_{i,j}$.

\begin{prop} The set of all $x_i$ and $S_{\ell}$ is a separating set for $\kk[V_2^{\oplus n}]^{C_p}$.
\end{prop}
\begin{proof}
It suffices to show that given the values of all $x_i$ and $f_{\ell}$, we may recover the values of each $u_{ij}$. We induce on $n$. If $n=1$, there is nothing to show.

\emph{Case 1: $x_n\neq 0$}: In this case,  we may write 
\begin{align}
 u_{i,i} &= x_n^{-1} ( x_i u_{n,n} + x_i^{p-1} x_n^{p-1} u_{i,n} + (-u_{i,n})^p\, )  \label{for-1}  \\ 
 u_{i,j} &= x_n^{-1} ( x_j u_{i,n} - x_i u_{j,n} )\,, \qquad i<j  \label{for-2}
\end{align}
to express each $u_{i,j}$ with $j<n$ in terms of the $x_s$ and $u_{k,n}$ with $k \gs j$. This enables us to express each $u_{i,j}$ in terms of the $S_{\ell}$ and $x_s$: indeed, $u_{n,n}=S_{2n}$, and if each $u_{i,j}$ with $j \gs k$ has such an expression, then 
\[ S_{n+k-1}=u_{k-1,n}+\sum\limits_{\substack{i+j=n+k-1 \\ j \gs k}} u_{i,j} \]
provides such an expression for $u_{k-1,n}$, and the formulas (\ref{for-1}) and (\ref{for-2}) above provide such an expression for $u_{k-1,k-1}$ and each $u_{k-1, j}$.

\emph{Case 2: $x_n = 0$}: Here, we have $y_n^p=u_{nn}$, so that $u_{i,n}=x_i y_n = x_i u_{n,n}^{1/p}$. Then, by the induction hypothesis, we may express each $u_{i,j}$ with $ j<n$ in terms of the $x_s$ and
\[\hat{S}_{\ell}= \sum\limits_{\substack{ i+j=\ell \\ j< n}} u_{i,j} = S_{\ell} - x_{\ell - n} u_{n,n}^{1/p} \]
(where $x_{\ell - n} : = 0$ for $\ell \ls n$), and thus in terms of the $x_s$ and $S_{\ell}$.
\end{proof}

As the action of $C_p$ on $V_2^{\oplus n}$ is generated by $n$-reflections, by Theorem~\ref{nonvanishing}, any separating set for $\kk[V_2^{\oplus n}]^{C_p}$ has at least $3n-1$ elements. Thus, the set 
\[\{x_i, S_{\ell} \mid {1 \ls i \ls n}, {2 \ls \ell \ls 2n}\}\]
 is a separating set of minimal size.

\subsection{A Non-Rigid Reflection Group} 
Let $\kk$ have characteristic 2 and $G$ be the finite subgroup of  $\GL_{7}(\FF_2)$ given by

\[G:=\left\{\renewcommand{\arraystretch}{0.7}
\renewcommand{\arraycolsep}{2pt}
\left(\begin{array}{c|c}
I_{4} & \mathbf{0}\\
\hline
\begin{array}{cccc}
\alpha_{1}&0&0&\alpha_{4}\\
0&\alpha_{2}&0&\alpha_4\\
0&0&\alpha_{3}&\alpha_{4}
\end{array}
&
I_{3}
\end{array}\right)\ \Big| \ \alpha_{1},\ldots,\alpha_{4}\in \FF_2 \right\},
\]
where  $I_{m}$ denotes the $m\times m$ identity matrix. The group $G$ is isomorphic to $C_2^4$, and generated by
reflections (namely those elements where exactly one of the $\alpha_i$'s is non-zero). This is a remarkable example since its invariant ring is not Cohen-Macaulay (see  \cite{KemperOnCM}) and, moreover, neither is any graded separating subalgebra (see  \cite{DEK:CMSep}) despite the action of $G$ being generated by reflections. 

Setting all $\alpha_i$'s to be 1 yields an element $\sigma$ whose fixed space of codimension $3$ is a minimal reflecting subspace. By Theorem~\ref{nonvanishing}, it follows that any separating set contains at least 9 elements. Writing $x_i$ for the coordinate functions on $V=\kk^{7}$, one has the minimal generating set
\[ \kk[V]^G=\kk[x_1, x_2, x_3, x_4, f_1, f_2, f_3, g_1, g_2, g_3, r] \]
where $\deg f_i=3$, $\deg g_i=4$, and $\deg r=5$. Using a computer algebra system, one verifies that
\begin{align}
\label{rig1} f_i r &\in \kk[x_1, x_2, x_3, x_4, f_1, f_2, f_3, g_1, g_2, g_3],\, \textrm{for} \quad i=1,2,3,\\
\label{rig2} r^2 &\equiv (x_1+x_4)^2 g_2 g_3\quad \mdd{f_1,f_2,f_3}\,.
\end{align}
Thus, given the values of the $x_i$'s, $f_i$'s, and $g_i$'s, one may recover the value of $r$ using \eqref{rig1} if some $f_i\neq 0$ and \eqref{rig2} if all $f_i=0$, so we can leave out $r$ still have a separating set. One also finds
\begin{align}
\label{rig3} (x_3 + x_4) f_3 &= f_2 (x_2+x_4)+ f_1 (x_1+x_4) &\\
\label{rig4} (x_i+x_4)^2 g_3 &\equiv f_i^2 \quad\mdd{x_3+x_4}, \quad i=1,2, &\\
\label{rig5} f_3 &\equiv 0 \quad\mdd{x_1+x_4,x_2+x_4,x_3+x_4}\,.
\end{align}
Hence, given the values of the $x_i$'s, $f_1,$ and $ f_2$, one can either obtain the value of $f_3$ (using \eqref{rig3} if $x_3\neq x_4$ or \eqref{rig5} if $x_1=x_2=x_3=x_4$) or $g_3$ (using \eqref{rig4} if $x_3=x_4$ and either $x_1\neq x_4$ or $x_2\neq x_4$). 
Concluding, we have the following:
\begin{prop} The invariants $x_1,  x_2, x_3, x_4, f_1, f_2, g_1, g_2, f_3 + g_3$ form a separating set for $\kk[V]^G$ of minimal size.
\end{prop}

\section*{Acknowledgements}
The authors thank MSRI, where most of the work on this project was completed. We also thank Dave Benson, Gregor Kemper, Anurag Singh, and Bernd Sturmfels for helpful conversations.


\end{document}